\documentclass[11pt]{article}
\usepackage{verbatim} 
\usepackage{graphicx}
\usepackage[utf8]{inputenc}
\usepackage{amsmath,amsfonts,amssymb,amscd,amsthm,txfonts,hyperref}
\usepackage{xcolor}
\newtheorem{theorem}{Theorem}[section]
\newtheorem{proposition}[theorem]{Proposition}

\newtheorem{remark}{Remark}[section]
\newtheorem{theo}{Theorem}

\theoremstyle{definition}
\newtheorem{definition}[theorem]{Definition}
\newtheorem{notation}[theorem]{Notation}

\newcommand{\R}{\mathbb{R}}

\newcommand{\N}{\mathbb{N}}

\newcommand{\T}{\mathbb{T}}

\newcommand{\E}{\mathbb E}

\newcommand{\ent}[1]{\lceil #1\rceil}

\begin{document}
\newcommand{\Hloc}{H^{1/2-}_{\textrm{loc}}}

\title{Continuity of the flow of the Benjamin-Bona-Mahony equation on probability measures}
\author{Anne-Sophie de Suzzoni\footnote{Universit\'e Paris 13, Sorbonne Paris Cit\'e, LAGA, CNRS ( UMR 7539), 99, avenue Jean-Baptiste Cl\'ement, F-93430 Villetaneuse, France}}

\maketitle

\begin{abstract}
 We use Wasserstein metrics adapted to study the action of the flow of the BBM equation on probability measures. We prove the continuity of this flow and the stability of invariant measures for finite times.
\end{abstract}

\tableofcontents

\section{Introduction}

The aim of this paper is to extend the result of \cite{dSwav} regarding the stability of Gaussian measures under the flow of the Benjamin-Bona-Mahony equation to more general measures. 

We consider the Benjamin-Bona-Mahony (BBM) equation on the torus $\T$:
$$
\partial_t (1- \partial_x^2 )u + \partial_x  (u + \frac{u^2}{2} ) = 0 \; .
$$
It follows from the work of Bona-Chen-Saut on Boussinesq equations \cite{BCSbouII,BCSbouI} that this equation is locally well-posed in $L^2$ and from the work of Bona-Tzvetkov \cite{BTsha} that it is globally well-posed in $H^s$, $s\geq 0$. 

We are interested in the action of this flow on measures. We consider measures on $H^s$, $s> 0$. The flow $\psi(t)$ of the BBM equation is well defined and continuous (hence measurable on the topological $\sigma$-algebra) on this space. Let $\rho$ be a measure on the topological $\sigma$-algebra of $H^s$, for all $t$, we can define the image $\rho^t$ of $\rho$ under the flow $\psi(t)$, namely, for all measurable subset $A$ of $H^s$,
$$
\rho^t(A) = \rho(\psi(t)^{-1}(A)) \; .
$$
We are interested in properties of the map $\rho \mapsto \rho^t$ at fixed $t$. We consider both its continuity and the stability of invariant measures under the flow $\psi(t)$, that is, measures such that $\rho^t = \rho$.

This study of invariant measures under the flow of partial differential equations is inspired from works by Lebowitz-Rose-Speers \cite{LRSstat}, Bourgain \cite{Bper}, and Zhidkov \cite{Zoninv}. From a physical point of view, the interest resides also in the evolution of the statistics described by more general measures, we can mention the works by Peierls \cite{Pzur}, Brout-Prigogine \cite{BPstat}, and Zakharov-Filonenko \cite{ZFweak}.

The topology that we use to prove the continuity of the map $\rho \mapsto \rho^t$ is the one induced by the Wasserstein metrics:
$$
d_{s',p}(\mu, \nu) = \inf_{\gamma \in \textrm{Marg}(\mu,\nu)} \left( \int_{H^s\times H^s } \|u-v\|_{H^{s'}}^p d\gamma(u,v) \right)^{1/p}\; ,
$$
where $\textrm{Marg} (\mu, \nu)$ is the set of measures on $H^s \times H^s$ whose marginals are $\mu$ and $\nu$, $s'\leq s$ corresponds to the regularity of the space where the measures can be compared and $p$ their integrability. In other words, given a large enough probability space $(\Omega, \mathcal A, \mathbb P)$ this distance can be seen as 
$$
\inf_{(X,Y) \in M(\mu, \nu)} \|X-Y\|_{L^p(\Omega, H^{s'})}
$$
where $M(\mu, \nu)$ is the set of couples of random variables $(X,Y) : \Omega \rightarrow H^s\times H^s$ such that the law of $X$ is $\mu$ and the one of $Y$ is $\nu$. This distance corresponds to the weak convergence of the measures combined to the convergence of the moments of order $q\leq p$:
$$
\int_{H^s} \|u\|_{H^{s'}}^q d\mu(u) \; .
$$

In \cite{CdScont}, the use of these distances is motivated.

We prove the following theorems.

\begin{theo}\label{theo-result1} Let $s\in ]1/4,1[$. Let $\mu, \nu$ such that for all $q\geq 1$, the moments of order $q$ of $\mu$ and $\nu$ satisfies:
$$
\left( \int_{H^s} \|u\|_{H^{s'}}^q d\mu(u) \right)\leq C_\mu \sqrt q \; , \left( \int_{H^s} \|u\|_{H^{s'}}^q d\nu(u) \right)\leq C_\nu \sqrt q
$$
with $C_\mu, C_\nu$ independent from $q$. Let $p \geq 1$ and $p_1$ and $p_2$ such that $1/p = 1/p_1 +1/p_2$. Then, for all $t$, and all $\sigma \in ]\max(\frac12,s), \min (1,2s)[$, we have 
$$
d_{0,p} (\mu^t,\nu^t) \leq C(\mu, \nu, t,p_1,\sigma) d_{s,p_2}(\mu, \nu)
$$
with
\begin{multline*}
C(\mu,\nu,t,p_1,\sigma) = C \big\| \Big( 1 + (CT \|u_{0,1}\|_{H^s})^{(\sigma - s)/s} \Big) e^{T\Big( 1 + (CT \|u_{0,1}\|_{H^s})^{\sigma /s}\Big)}\big\|_{L^{2p_1}_\mu} \times \\
 \big\| \Big( 1 + (CT \|u_{0,2}\|_{H^s})^{(\sigma - s)/s} \Big) e^{T\Big( 1 + (CT \|u_{0,2}\|_{H^s})^{\sigma /s}\Big)}\big\|_{L^{2p_1}_\nu} \; .
\end{multline*}
where $T = 1+|t|$ and 
$$
\|F(u_{0,1})\|_{L^q_\mu} = \Big(\int_{H^s} F(u_{0,1})^q d\mu(u_{0,1}) \Big)^{1/q} \; ,\;  \|F(u_{0,2})\|_{L^q_\nu} = \Big(\int_{H^s} F(u_{0,2})^q d\mu(u_{0,2}) \Big)^{1/q}\; .
$$
\end{theo}

In other words, $\rho \mapsto \rho^t$ is locally Lipschitz continuous for the distances $d_{0,p}$ and $d_{s,p_2}$ in the set of measures satisfying certain constraints on their moments.

\begin{theo}\label{theo-result2}Let $s\in]1/3,1[$. Let $\rho$ be an invariant measure under the flow of BBM such that there exists $C$ such that for all $q\geq 1$,
$$
\left( \int \|u\|_{H^s}^qd\rho(u) \right)^{1/q} \leq C\sqrt q \; .
$$
Let $p_0 > 1$ and let $\mu $ be a measure on $H^s$ such that 
$$
\int \|u\|_{H^s}^{p_0} d\mu(u) < \infty \; .
$$
Let $p_1,p_2 \leq p_0$ and $p$ such that $1/p = 1/p_1 +1/p_2 \leq 1$. Then, for all $t \in \R$, all $\sigma \in ]max(s,\frac12) , \min (1, \frac23 s)[$
$$
d_{0,p} (\mu^t, \rho) \leq C(p_1,t,\mu, \sigma) d_{s,p_2} (\mu,\rho)
$$
where
$$
C(p_1,t,\mu,\sigma) =  C \Big( 1 + C T^{(\sigma - s)/s} ( \sqrt{\frac{p_1 (\sigma -s )}{s}} + \|u_{0,2}\|_{L^{2p_1(\sigma -s)/s},H^s}^{(\sigma -s)/s}\Big)  e^{c T^2 p_1}
$$
and $T = 1+|t|$.

In particular, if there exists $C$ such that for all $q\geq 1$, 
$$
d_{s,q} (\mu, \rho) \leq C^q \varepsilon \; .
$$
then for all $t$, there exists $C(t, \mu)$ such that
$$
d_{0, q} (\mu^t,\rho) \leq C(t, \mu)^q \varepsilon \; .
$$
In other words $\rho$ is stable in the set of measures $\mu$ such that there exists $\delta > 0$ satisfying
$$
 1_{\|u\|_{H^s}\geq 1}e^{\delta (\ln \|u\|_{H^s})^2} \in L^1_\mu \; .
$$
\end{theo}

\begin{remark} In \cite{dSwav}, an invariant measure $\rho_0$ on $H^s$ with $s< \frac12$ is built. This measure is a Gaussian random variable on $H^s$ whose covariance operator is $L^2 = (1-\partial_x^2)^{-1}$. This measure is perturbed by considering other Gaussian variables on $H^s$ with covariance operators $L(1+V)L$ with $V \ll 1$. It is proved that $\rho_0$ is stable in the set of such Gaussian variables with $V$ small enough in the topology corresponding to the weak convergence of measures. Here, we consider a bigger set of measures - any measure with a finite $p_0$ moment - , a stronger topology for the measures, and the proof does not use the same tools. \end{remark}

\begin{remark} The loss of derivative in the result of continuity is needed in order to make the constant $C(\mu,\nu,t,p_1,\sigma)$ finite. In the result of stability (Theorem \ref{theo-result2}), it can be replaced by a stronger hypothesis on the integrability of $\mu$. Namely, in Theorem \ref{theo-result2}, it is possible to take any $s>0$ provided that $p_0 \geq 2p_1(\sigma -s)/s$. The loss in integration (the fact that $p>p_2$) in both result is due to H\"older inequalities and we do not know how it could be avoided. \end{remark}

\begin{remark} The result should be better considering measures on $H^s$ with $s\geq 1$ thanks to the invariance of the $H^1$ norm under the flow of BBM. The reason why we consider $s<1$ is that we want to discuss the stability of the known invariant measure, which is on $H^s$, $s<1/2$. We do not know invariant measures of BBM defined on $H^1$, except for trivial ones, such as the one concentrated on $0$.\end{remark}

The proof of these results consists in proving a deterministic global control, Proposition \ref{prop-diff}, on the $L^2$ norm of the difference between two solutions of BBM,
$$
\psi(t)u_{0,1} - \psi(t)u_{0,2} \; ,
$$
and then integrate the obtained inequality on the probability space where $u_{0,1}$ is $\rho$ or $\nu$ typical, and $u_{0,2}$ is $\mu$ typical. 

To prove the stability theorem, we use the invariance in the proof, which makes the result better in terms of hypothesis on $\mu$ than the continuity one.

\paragraph{Organisation of the paper}

In Section 2, we prove global estimates on $\psi(t)u_0$ and $\psi(t)u_{0,1} - \psi(t) u_{0,2}$. 

In Section 3, we define the space of measures on which we prove Theorems \ref{theo-result1}, \ref{theo-result2} and give alternative definitions or point of views of these spaces using large deviation estimates. Then, we prove Theorems \ref{theo-result1}, \ref{theo-result2}.

\section{Deterministic estimates}

Through all this paper, we use the fact that the BBM equation is locally well-posed in $L^2$ according to the following proposition, that comes from \cite{BCSbouII,BCSbouI}.

\begin{proposition}[from \cite{BCSbouII,BCSbouI}]\label{prop-lwp} There exists $C$ such that for all $\Lambda > 0$, the BBM equation is well-posed in $\mathcal C ([-T,T], L^2(\T))$ with $T= \frac1{C\Lambda}$ for initial data $u_0$ such that $\|u_0\|_{L^2} \leq \Lambda$. In particular, calling $u_{1}$ and $u_2$ the unique solutions with respective initial datum $u_{0,1}, u_{0,2}$, we have 
$$
\|u_i\|_{L^\infty([-T,T],L^2(\T))} \leq C \Lambda \mbox{ and }  \|u_1 - u_2 \|_{L^\infty([-T,T],L^2(\T))}\leq C \|u_{0,1} - u_{0,2}\|_{L^2}\; .
$$
\end{proposition}

Besides, we also use the fact that BBM is globally well-posed in $H^s$, as was proved in \cite{BTsha} .

\subsection{On a solution of BBM}

\begin{proposition}\label{prop-globonu} Let $T\geq 1$ and $u_0 \in H^s$. For all $N \in \N$ such that $N\geq (CT \|u_0\|_{H^s})^{1/s}$, and all $\sigma > \frac12$, we have that for all $t \in [-T,T]$, the solution $u$ of BBM with initial datum $u_0$ satisfies
$$
\|u(t)\|_{L^2} \leq C \Big( \frac1{T} + N^{\sigma - s }\|u_0\|_{H^s}\Big) 
$$
with $C$ independent from $u_0$ and $T$.
\end{proposition}

\begin{notation} We call $\Pi_N$ the orthogonal projection on 
$$
\mbox{Vect }\Big( \cos(nx) , n=0,\hdots, N, \sin(nx) , n=1 , \hdots, N\Big) \; .
$$
\end{notation}

\begin{proof} Set $N$ such that $N\geq (CT \|u_0\|_{H^s})^{1/s}$. We have that
$$
\|(1-\Pi_N) u_0 \|_{L^2} \leq N^{-s} \|u_0\|_{H^s} \leq \frac1{CT} \; .
$$
We can apply the local well-posedness proposition (Proposition \ref{prop-lwp}) for the initial datum $v_0 = (1-\Pi_N) u_0$. There exists a unique solution $v$ in $[-T,T] $ of BBM with initial datum $v_0$. Besides, $v$ satisfies for all $t\in [-T,T]$;
$$
\|v(t)\|_{L^2} \leq C' \frac1{CT} \lesssim \frac1T \; .
$$
Writing $u$ the solution of BBM with initial datum $u_0$, we call $w = u-v$. This function satisfies 
$$
\partial_t (1-\partial_x^2) w + \partial_x \Big( w + vw + \frac{w^2}{2} \Big)
$$
with initial datum $w_0 = \Pi_N u_0$. As $w_0$ is in $H^1$, it has been proved in \cite{BTsha} that $w \in H^1$ for the times $[-T,T]$. We compute estimates on $\|w\|_{H^\sigma}$ with $\sigma \in ]1/2,1]$. We start by differentiating $\|w\|_{H^1}^2$ with respect to time :
\begin{eqnarray*}
\partial_t \|w\|_{H^1}^2  &= &\partial_t \int w (1-\partial_x^2) w \\
 &=& 2 \int w \partial_t (1-\partial_x^2) w \\
 & = & -2 \int w\partial_x \Big( w + vw + \frac{w^2}{2}\Big) \; .
\end{eqnarray*}
Using that $\int w\partial_x w = 0$ and $\int w \partial_x w^2 = 0$, we keep only the term $\int w \partial_x (vw)$. Because $\partial_x$ is skew symmetric, we have 
\begin{eqnarray*}
2 \|w\|_{H^1}\partial_t \|w\|_{H^1} &=& 2 \int (\partial_x w) vw \\
\partial_t \|w\|_{H^1} & \leq & \|w\|_{L^\infty} \|v\|_{L^2} \; .
\end{eqnarray*}
Using the Sobolev embedding $H^\sigma \subset L^\infty$, ($\sigma > 1/2$) , the fact that $\sigma $ is less than $1$, and that $\|v\|_{L^2} \lesssim T^{-1}$, we get
$$
\|w(t)\|_{H^\sigma} \leq \|w(t)\|_{H^1} \lesssim \frac1{T} \int_{0}^t \|w(\tau)\|_{H^\sigma} \; .
$$
Using Gronwall lemma, we get
$$
\|w(t)\|_{H^\sigma} \leq e^{c\frac{|t|}{T}} \|w_0\|_{H^\sigma} \; .
$$
We use that $|t| \leq T$ and $\|w_0\|_{H^\sigma} \leq N^{\sigma-s}\|u_0\|_{H^s}$ to conclude :
$$
\|u(t)\|_{L^2} \leq \|v(t)\|_{L^2} + \|w(t)\|_{H^\sigma} \leq \frac{1}{CT} + C N^{\sigma-s} \|u_0\|_{H^s} \; .
$$
\end{proof}

\subsection{On the difference of two solutions}

In this subsection, we estimate the difference between two solutions of BBM with the difference between the initial datum.

\begin{proposition}\label{prop-diff}Let $T \geq 1$, $u_{0,1} \in H^s$, $u_{0,2} \in H^s$ and $\sigma \in ]1/2,1]$. We call $u_i$ the solution of BBM with initial datum $u_{0,i}$. We have that for all $t \in [-T,T]$, $u_1-u_2$ satisfies 
$$
\|u_1(t) - u_2(t)\|_{L^2} \leq C \Big( 1 + (CT \|u_{0,1}\|_{H^s})^{(\sigma - s)/s} + (CT \|u_{0,2}\|_{H^s})^{(\sigma - s)/s}\Big) e^{c\int_{0}^t \|u_1(\tau)\|_{L^2}d\tau} \|u_{0,1}-u_{0,2}\|_{H^s} 
$$
with constants $C$ and $c$ independent from $u_0$ and $T$.
\end{proposition}

\begin{proof} Let 
$$
N = \max (\ent{(CT\|u_{0,1}\|_{H^s})^{1/s} } , \ent{(CT\|u_{0,2}\|_{H^s})^{1/s}})
$$
where $\ent x$ is the smallest integer $n$ such that $n\geq x$. We call $v_{0,i} = (1-\Pi_N ) u_{0,i}$, we have that $\|v_{0,i}\|_{L^2} \leq \frac1{CT}$. Hence, we can apply Proposition \ref{prop-lwp} to $v_{0,i}$, $i=1,2$. There exist two unique solutions of BBM on the times $[-T,T]$, $v_1$ and $v_2$ with respective initial data $v_{0,1}$ and $v_{0,2}$, and besides, $v_1$ and $v_2$ satisfy
$$
\|v_i\|_{L^2} \lesssim T^{-1} \; , \; \|v_1 - v_2\|_{L^2} \lesssim \|v_{0,1}-v_{0,2}\|_{L^2} \leq N^{-s} \|u_{0,1}-u_{0,2}\|_{H^s} \; .
$$
Let $w_i = u_i - v_i$, we have that $w_i$ is the solution of 
$$
\partial_t (1-\partial_x^2) w_i + \partial_x \Big( w_i + v_i w_i + \frac{w_i^2}{2} \Big) = 0
$$
with initial datum $w_{0,i} = \Pi_N u_{0,i}$.

We set $w= w_1-w_2$, $w$ satisfies 
$$
\partial_t (1-\partial_x^2) w + \partial_x \Big( w + v_1w_1 -v_2w_2 + \frac{w_1^2-w_2^2}{2} \Big) =0 
$$
with initial datum $w_0 = w_{0,1} -w_{0,2}$. We can write this equation 
$$
\partial_t (1-\partial_x^2) w + \partial_x \Big( w + wu_1 + (v_1 - v_2) w_2- \frac{w^2}{2} \Big) = 0 \; .
$$
Indeed, we have 
\begin{eqnarray*}
v_1 w_1 - v_2 w_2  &=& v_1 w + v_1 w_2 - v_2 w_2  \\
 & = & v_1 w + (v_1-v_2)w_2 
\end{eqnarray*}
and
\begin{eqnarray*}
w_1^2 - w_2^2 &=& w (w_1+w_2) \\
 &= & w (2 w_1 - w) \\
 &=& 2w_1 w - w^2 
\end{eqnarray*}
and by summing this equalities
\begin{eqnarray*}
 v_1w_1 -v_2w_2 + \frac{w_1^2-w_2^2}{2} &=& v_1 w + w_1 w + (v_1 - v_2) w_2 - \frac{w^2}{2} \\
 &=& u_1 w + (v_1 - v_2) w_2 - \frac{w^2}{2} \; .
\end{eqnarray*}

Let us estimate $w$. We write $v_1-v_2 = v$. We have 
\begin{eqnarray*}
\partial_t \|w\|_{H^1}^2 &=& \partial_t \int w (1-\partial_x^2) w  \\
 &=& 2 \int w \partial_t(1-\partial_x^2) w \\
&=& -2 \int w \partial_x \Big( w + wu_1 + v w_2- \frac{w^2}{2} \Big) 
\end{eqnarray*}
and by keeping only the non null terms, we get
$$
\partial_t \|w\|_{H^1}^2 = 2 \int (\partial_x w) ( wu_1 + v w_2 ) \; .
$$
Thus, we have
$$
2 \|w\|_{H^1} \partial_t \|w\|_{H^1} \leq 2 \|w\|_{H^1} (\|w\|_{L^\infty} \|u_1\|_{L^2} + \|v\|_{L^2} \|w_2\|_{L^\infty})\; .
$$
Using Sobolev embedding $H^\sigma \in L^\infty$, we get
$$
\partial_t \|w\|_{H^1} \leq C (\|w\|_{H^\sigma} \|u_1\|_{L^2} + \|v\|_{L^2} \|w_2\|_{H^\sigma})\; .
$$
By integrating over time, we get
$$
\|w(t)\|_{H^\sigma } \leq \|w(t)\|_{H^1} \leq C \int_{0}^t (\|w(\tau)\|_{H^\sigma} \|u_1(\tau)\|_{L^2} + \|v(\tau)\|_{L^2} \|w_2(\tau)\|_{H^\sigma})d \tau
$$
and by using Gronwall lemma
$$
\|w(t)\|_{H^\sigma} \leq e^{c\int_{0}^t \|u_1(\tau)\|_{L^2} d\tau} \|w_0\|_{H^\sigma} + \int_{0}^t  e^{c\int_{t'}^t \|u_1(\tau)\|_{L^2} d\tau}\|v(t')\|_{L^2} \|w_2(t')\|_{H^\sigma}dt'\; .
$$
We estimate each term. Thanks to Proposition \ref{prop-lwp}, we have 
$$
\|v(t')\|_{L^2} \leq N^{-s} \|u_{0,1}-u_{0,2}\|_{H^s}\; .
$$
Thanks to Proposition \ref{prop-globonu}, we have 
$$
\|w_2(t')\|_{H^\sigma} \leq C N^{\sigma - s}\|u_{0,2}\|_{H^s}\; .
$$
By definition of $N$, $N^s \geq CT \|u_{0,2}\|_{H^s}$, thus  
$$
\|w_2(t')\|_{H^\sigma} \leq C N^\sigma T^{-1} 
$$
hence 
$$
\|v(t')\|_{L^2} \|w_2(t')\|_{H^\sigma} \leq C T^{-1} N^{\sigma - s} \|u_{0,1}-u_{0,2}\|_{H^s} \; .
$$
Finally, we use that 
$$
e^{c\int_{t'}^t \|u_1(\tau)\|_{L^2} d\tau} \leq e^{c\int_{0}^t \|u_1(\tau)\|_{L^2} d\tau}
$$
to get
$$
\int_{0}^t dt' e^{c\int_{t'}^t \|u_1(\tau)\|_{L^2} d\tau}\|v(t')\|_{L^2} \|w_2(t')\|_{H^\sigma} \leq C e^{c\int_{0}^t \|u_1(\tau)\|_{L^2} d\tau}  N^{\sigma - s} \|u_{0,1}-u_{0,2}\|_{H^s}\; .
$$
The initial datum $w_0 = \Pi_N (u_{0,1} - u_{0,2})$ satisfies 
$$
\|w_0\|_{H^\sigma} \leq N^{\sigma -s}\|u_{0,1}-u_{0,2}\|_{H^s} \; .
$$
Therefore, we have the inequality
$$
\|w(t)\|_{H^\sigma} \leq C e^{c\int_{0}^t \|u_1(\tau)\|_{L^2} d\tau}  N^{\sigma - s} \|u_{0,1}-u_{0,2}\|_{H^s} \; .
$$
We estimate $N$. By definition, $N$ is less than
$$
\max (1+ (CT \|u_{0,1}\|_{H^s})^{1/s}, 1+ (CT \|u_{0,1}\|_{H^s})^{1/s}) 
$$
hence 
$$
N^{\sigma - s} \leq C(s,\sigma) \Big( 1 + (CT \|u_{0,1}\|_{H^s})^{(\sigma - s)/s} + (CT \|u_{0,2}\|_{H^s})^{(\sigma - s)/s}\Big) \; ,
$$
which yields
$$
\|w(t)\|_{H^s} \leq C \Big( 1 + (CT \|u_{0,1}\|_{H^s})^{(\sigma - s)/s} + (CT \|u_{0,2}\|_{H^s})^{(\sigma - s)/s}\Big) e^{c\int_{0}^t \|u_1(\tau)\|_{L^2}d\tau} \|u_{0,1}-u_{0,2}\|_{H^s} 
$$
and since $u_1 - u_2 = v + w$ and the $L^2$ norm of $v$ is less than $N^{-s}\|u_{0,1}-u_{0,2}\|_{H^s}$ which is less than the above bound, we have proved the proposition. \end{proof}

\section{Probabilistic integration}

\subsection{Definitions and large deviation estimates}

In this subsection, we define the spaces of probability measures where we prove the continuity and stability, along with distances on these spaces, and we prove the equivalence between large deviation estimates and estimates on the moments of these measures.

\paragraph{Continuity}

\begin{notation} Let $\mathcal M(H^s)$ be the set of probability measures on the topological $\sigma$-algebra of $H^s(\T)$. \end{notation}

The space where we prove the continuity is the one of the measures  with large Gaussian deviation estimates.

\begin{notation} Let $\Sigma$ be the set of probability measures in $\mathcal M(H^s)$ with large Gaussian deviation estimates, that is :
$$
\Sigma  = \Big\{ \rho \in \mathcal M(H^s) \; \Big| \; \exists \delta > 0 \; ; \; \int_{H^s} e^{\delta \|u\|_{H^s}^ 2} d\rho(u) < \infty \Big\} \; .
$$
\end{notation}

We have an equivalence between belonging to $\Sigma$ and satisfying estimates on the moments of order $p$.

\begin{proposition}\label{prop-glde} A measure $\rho \in \mathcal M(H^s)$ belongs to $\Sigma$ if and only if there exists $C(\rho)$ such that for all $p \geq 1$, $\|\; \|u\|_{H^s} \|_{L^p(d\rho(u))} \leq C(\rho) \sqrt p$. \end{proposition}

\begin{notation} We write for all $F : H^s \rightarrow \R$ measurable
$$
\|F(u)\|_{L^p_\rho} = \|F(u)\|_{L^p(d\rho(u))} = \Big( \int_{H^s} |F(u)|^p d\rho(u)\Big)^{1/p}\; .
$$
\end{notation}

\begin{proof} This is a well-known property hence we only sketch the proof. For more details, we refer to Proposition 4.4 of \cite{BTranI} .

Assume that $\rho \in \Sigma$. Let $X = \|u\|_{H^s}$. We have, thanks to Markov's inequality
$$
\rho( X \geq \lambda ) \leq e^{-\delta \lambda^2} \E_\rho ( e^{\delta X^2})
$$
where $\E_*$ is the average with regard to the measure $*$. Hence, we get
$$
\E_\rho ( X^p) = \int p \lambda^{p-1} \rho( X \geq \lambda) d\lambda \leq C(\rho) \int p \lambda^{p-1} e^{-\delta \lambda^2}d\lambda \; .
$$
By using the change of variable $\lambda  = \sqrt{1/2 \delta}y$, we get
$$
\E_\rho ( X^p) \leq C \Big( \frac1{\sqrt{2\delta}}\Big)^p \int_{0}^\infty p y^{p-1} e^{-y^2/2} dy \; .
$$
The integral $\int_{0}^\infty p y^{p-1} e^{-y^2/2} dy$ does not depend on $\rho$ and by induction we have that it is less than $C p^{p/2}$, hence 
$$
\E_\rho ( X^p) \leq C \Big( \frac1{\sqrt{2\delta}}\Big)^p p^{p/2}
$$
therefore
$$
\|X\|_{L^p_\rho} = \E_\rho (X^p)^{1/p} \leq C \sqrt p \; .
$$

Conversely, assume that 
$$
\|X\|_{L^p_\rho} \leq C \sqrt p \; .
$$
Then, the probability $\rho(X\geq \lambda)$ can be bounded by
$$
\rho(X\geq \lambda) = \rho(X^p\geq \lambda^p) \leq \lambda^{-p} \E_\rho(X^p) \leq \Big( \frac{C\sqrt p}{\lambda}\Big)^p\; .
$$
By choosing $p$ such that $ \frac{C\sqrt p}{\lambda} = e^{-1}$ that is $p = e^{-2}\lambda^2/C^2$, we get
$$
\rho(X\geq \lambda) \leq e^{- c \lambda^2}
$$
which ensures that $e^{\delta X^2}$ is integrable for all $\delta < c$.
\end{proof}

\paragraph{Stability}

For the stability, the hypothesis on the measures is weaker, we only assume that it has a $p$-moment in $H^s$.

\begin{notation} We call $\Sigma_p$ the measures on $H^s$ with a $p$-moment ($p\geq 1$), that is : 
$$
\Sigma_p = \Big\{ \rho \in \mathcal M(H^s) \; \Big| \; \E_{\rho} (\|u\|_{H^s}^p) < \infty \Big\} \; .
$$
\end{notation}

To compare measures, we use the Wasserstein metrics.

\begin{definition} Let $s'\leq s$ and $p'\leq p$, let $\mu$, $\nu$ two measures in $\Sigma_p$. The Wasserstein metrics $d_{s',p}$ is defined as 
$$
d_{s',p'}(\mu,\nu) = \inf_{\gamma \in \textrm{Marg} (\mu,\nu)} \Big( \int \|u_1 - u_2\|_{H^{s'}}^{p'} d\gamma( u_1,u_2) \Big)^{1/p'}
$$
where $\textrm{Marg} (\mu,\nu)$ is the set of probability measures on $H^s \times H^s$ whose marginals are $\mu$ and $\nu$, that is, for all $A$ measurable in $H^s$, $\gamma(A\times H^s) = \mu(A)$ and $\gamma (H^s\times A) = \nu (A)$.
\end{definition}

We will compare the measures transported by the BBM flow in $d_{0,p'}$ using the $d_{s,p}$ distance for the initial data.

\paragraph{Another large deviation estimate}

\begin{proposition}\label{prop-alde}Let $X$ be a random variable on a probability space with measure $\rho$ and let $\alpha > 0$. The fact that there exists $\delta > 0$ such that $e^{\delta (\ln X)^{1/\alpha +1}}1_{X\geq 1}$ is $\rho$-integrable is equivalent to the fact that there exists $E_0$ and $C> 1$ such that for all $p$
$$
\E_\rho (X^p) \leq E_0 C^{p^{1+\alpha}} \; .
$$
\end{proposition}

\begin{proof} We assume that 
$$
\E_\rho (X^p) \leq E_0 C^{p^{1+\alpha}}\; .
$$
For $\lambda \geq 1$, thanks to Markov's inequality, we have 
$$
\rho( X \geq \lambda )  = \rho(X^p \geq \lambda^p) \leq E_0 C^{p^{1+\alpha}} \lambda^{-p} \; .
$$
We minimize 
$$
f(p) = C^{p^{1+\alpha}} \lambda^{-p}
$$
by taking its logarithm,
$$
\ln f(p) = p^{1+ \alpha} \ln C -p \ln \lambda
$$
and differentiating it 
$$
(\ln f)'(p) = (1+ \alpha) p^\alpha \ln C - \ln \lambda\; .
$$
We get that $f$ is minimal when 
$$
p= p_0 := \Big( \frac{\ln \lambda }{\ln C (1+ \alpha)}\Big)^{1/ \alpha} \; ,
$$
that is 
$$
\min f = e^{- \frac{(\ln \lambda)^{1/\alpha+1}}{(\ln C)^{1/\alpha}}\beta(\alpha)}
$$
with $\beta(\alpha) = (1+\alpha)^{-1/\alpha } (1- \frac1{1+ \alpha}) > 0$. Therefore,
$$
\rho(X\geq \lambda) \leq E_0 e^{- \frac{(\ln \lambda)^{1/\alpha+1}}{(\ln C)^{1/\alpha}}\beta(\alpha)}\; .
$$
Since 
$$
\E_\rho (e^{\delta (\ln X)^{1+1/\alpha}}1_{X\geq 1}) = \int_{1}^\infty \delta \Big( 1+ \frac1{\alpha} \Big)\frac{(\ln \lambda)^{1/\alpha}}{\lambda}e^{\delta (\ln \lambda)^{1+1/\alpha}}\rho(X\geq \lambda) d\lambda \; ,
$$
with the change of variable $x= \ln \lambda$, we have 
$$
\E_\rho (e^{\delta (\ln X)^{1+1/\alpha}}1_{X\geq 1}) \leq \int_{0}^\infty \delta (1+ \frac1\alpha ) x^{1/\alpha} e^{\Big(\delta - \frac{\beta(\alpha)}{(\ln C)^{1/\alpha}}\Big)x^{1+1/\alpha}} dx
$$
which ensures that it is finite as long as $\delta$ is strictly less than $\frac{\beta(\alpha)}{(\ln C)^{1/\alpha}}$.

Conversely, we assume that 
$$
\E_\rho (e^{\delta (\ln X)^{1+1/\alpha}}1_{X\geq 1}) = E_1 < \infty \; .
$$
Then, the probability 
$$
\rho( X\geq \lambda)
$$
is less than $1$ if $\lambda \leq 1$ and is less than 
$$
E_1 e^{-\delta (\ln \lambda)^{1+1/\alpha}}
$$
otherwise. Hence, for $p\geq 1$, we get that
$$
\E_\rho (X^p) \leq 1 + E_1 \int_{1}^{\infty} p \lambda^{p-1} e^{-\delta (\ln \lambda)^{1/\alpha+1}}d\lambda\; .
$$
Let 
$$
I = \int_{1}^{\infty} p \lambda^{p-1} e^{-\delta (\ln \lambda)^{1/\alpha+1}}d\lambda\; ,
$$
and by writing $\lambda^{p-1} = e^{(p-1)\ln \lambda}$,  
$$
I = \int_{1}^{\infty} p  e^{-\delta (\ln \lambda)^{1/\alpha+1}+(p-1)\ln \lambda}d\lambda\; .
$$
We have that
$$
(p-1) \ln \lambda \leq \frac\delta2 (\ln \lambda)^{1/\alpha +1}
$$
if and only if 
$$
\ln \lambda \geq \ln \lambda_0 = \Big( \frac{2(p-1)}{\delta} \Big)^{\alpha}\; .
$$
Hence, we have, by dividing the integration between $[1,\lambda_0]$ and $[\lambda_0,\infty[$
$$
I \leq I.1 +I.2 = p\lambda_0^p + \int_{\lambda_0}^\infty pe^{-\delta (\ln \lambda)^{1/\alpha +1}/2} d\lambda\; .
$$
The quantity $I.2$ is less than
$$
 \int_{1}^\infty pe^{-\delta (\ln \lambda)^{1/\alpha +1}/2} d\lambda\; .
$$
We have that
$$
\delta (\ln \lambda)^{1/\alpha +1}/2 \geq 2 \ln \lambda
$$
if and only if 
$$
\ln \lambda \geq \ln \lambda_1 = \Big( \frac4\delta \Big)^{\alpha}\; .
$$
Therefore, we get
$$
I.2 \leq p \lambda_1 + \int_{\lambda_1}^\infty p e^{-2\ln \lambda} d\lambda  \leq p\lambda_1 + p \int_{1}^\infty \frac{d\lambda}{ \lambda^2}\; .
$$
As $\lambda_1$ does not depend on $p$, we get that 
$$
I.2 \leq C p
$$
where $C$ depends on $X$ and $\alpha$. For $I.1$, we use that
$$
\lambda_0 = e^{ \Big( \frac{2(p-1)}{\delta} \Big)^{\alpha}}\leq C^{p^\alpha}
$$
hence 
$$
I.1 = p \lambda_0^p \leq C^{p^{1+\alpha}}
$$
and by summing $I.1$ and $I.2$,
$$
I \leq C^{p^{1+\alpha}}\; .
$$
Finally, as 
$$
\E_\rho (X^p) \leq 1 + E_1 I \; ,
$$
we get 
$$
\E_\rho (X^p) \leq E_0 C^{p^{1+\alpha}}
$$
which concludes the proof. \end{proof}

\subsection{Continuity of the flow}

In this subsection, we prove the continuity of the action of the flow of BBM.

\begin{definition} Let $\mu \in \mathcal M(H^s)$. For all $t\in \R$ we call $\mu^t$ the image measure of $\mu$ under the flow of BBM $\psi (t)$, that is, for all measurable set $A$,
$$
\mu^t(A) = \mu (\psi(t)^{-1}(A)) \; .
$$
\end{definition}

\begin{remark} This definition is possible because $\psi(t)$ is continuous on $H^s$ and hence measurable on its topological $\sigma$-algebra.\end{remark}

\begin{proposition}\label{prop-cont} Let $s\in]1/4,1[$, $p \geq 1$ and $p_1,p_2$ such that $\frac1{p_1} + \frac1{p_2}  = \frac1{p} $ and $t\in \R$. Let $T = 1+|t|$. Let $\sigma \in ]1/2,1[$ such that $1<\sigma/s  < 2$. For all $\mu,\nu \in  \Sigma$, we have 
$$
d_{p,0} (\mu^t, \nu^t) \leq C(\mu, \nu, t, p_1,\sigma) d_{p_2,s} (\mu, \nu) 
$$
with 
\begin{multline*} 
C(\mu,\nu,t,p_1,\sigma) = C \| \Big( 1 + (CT \|u_{0,1}\|_{H^s})^{(\sigma - s)/s} \Big) e^{T\Big( 1 + (CT \|u_{0,1}\|_{H^s})^{\sigma /s}\Big)}\|_{L^{2p_1}_\mu} \times \\
 \| \Big( 1 + (CT \|u_{0,2}\|_{H^s})^{(\sigma - s)/s} \Big) e^{T\Big( 1 + (CT \|u_{0,2}\|_{H^s})^{\sigma /s}\Big)}\|_{L^{2p_1}_\nu} \; .
\end{multline*}
\end{proposition}

\begin{proof} Let $\gamma \in \textrm{Marg} (\mu, \nu)$, that is, $\gamma$ is a measure on $H^s\times H^s$ whose marginals are $\mu$ and $\nu$. Set $\gamma^t$ the image measure of $\gamma$ under the map $(\psi(t), \psi(t))$. For all $A$ measurable in $H^s$, we have 
$$
\gamma^t (A \times H^s) = \gamma (\psi(t)^{-1} A\times \psi(t)^{-1} H^s) = \gamma (\psi(t)^{-1} (A) \times H^s) \; .
$$
Since the marginals of $\gamma$ are $\mu$ and $\nu$, we get
$$
\gamma^t (A \times H^s) =\mu (\psi(t)^{-1} (A)) = \mu^t (A)\; .
$$
For the same reasons,
$$
\gamma^t(H^s \times A) = \nu^t (A) \; .
$$
In other words, $\gamma^t \in \textrm{Marg} (\mu^t,\nu^t)$. Therefore, we get
$$
d_{p, 0} (\mu^t, \nu^t) \leq \left( \int \|u_{1} - u_{2}\|_{L^2}^p d\gamma^t(u_{1}, u_{2}) \right)^{1/p} \; .
$$
We do the change of variable $(u_1,u_2) = (\psi(t) u_{0,1},\psi(t)u_{0,2}) = (\psi(t),\psi(t))(u_{0,1},u_{0,2})$, we get, thanks to the definition of $\gamma^t$,
$$
d_{p, 0} (\mu^t, \nu^t) \leq \left( \int \|\psi(t)u_{0,1} - \psi(t)u_{0,2}\|_{L^2}^p d\gamma(u_{0,1}, u_{0,2}) \right)^{1/p}\; .
$$
We set $u_i(t) = \psi(t) u_{0,i}$. We input the estimate of Proposition \ref{prop-diff} with $T =1+ |t|$ ($T$ has to be bigger than $1$),
$$
\|u_1(t) - u_2(t)\|_{L^2} \leq C \Big( 1 + (CT \|u_{0,1}\|_{H^s})^{(\sigma - s)/s} + (CT \|u_{0,2}\|_{H^s})^{(\sigma - s)/s}\Big) e^{c\int_{0}^t \|u_1(\tau)\|_{L^2}d\tau} \|u_{0,1}-u_{0,2}\|_{H^s} \; ,
$$
we get
$$
d_{p, 0} (\mu^t, \nu^t) \leq  C \big\|\Big( 1 + (CT \|u_{0,1}\|_{H^s})^{(\sigma - s)/s} + (CT \|u_{0,2}\|_{H^s})^{(\sigma - s)/s}\Big) e^{c\int_{0}^t \|u_1(\tau)\|_{L^2}d\tau} \|u_{0,1}-u_{0,2}\|_{H^s}\big\|_{L^p_\gamma} \; .
$$
We use that $1/p = 1/p_1 + 1/p_2$ to write the H\"older inequality,
$$
d_{p, 0} (\mu^t, \nu^t) \leq  C \big\|\Big( 1 + (CT \|u_{0,1}\|_{H^s})^{(\sigma - s)/s} + (CT \|u_{0,2}\|_{H^s})^{(\sigma - s)/s}\Big) e^{c\int_{0}^t \|u_1(\tau)\|_{L^2}d\tau}\big\|_{L^{p_1}_\gamma} \big\| \|u_{0,1}-u_{0,2}\|_{H^s}\big\|_{L^{p_2}_\gamma} \; .
$$
Let
$$
I  = \|\Big( 1 + (CT \|u_{0,1}\|_{H^s})^{(\sigma - s)/s} + (CT \|u_{0,2}\|_{H^s})^{(\sigma - s)/s}\Big) e^{c\int_{0}^t \|u_1(\tau)\|_{L^2}d\tau}\|_{L^{p_1}_\gamma}\; .
$$
We use the estimates on $u_1$ in Proposition \ref{prop-globonu} 
$$
\|u_1(t)\|_{L^2} \leq C \Big( \frac1{T} + N^{\sigma - s }\|u_{0,1}\|_{H^s}\Big) 
$$
and that $N$ is build such that 
$$
N^{\sigma-s} \lesssim 1 + (CT \|u_{0,1}\|_{H^s})^{(\sigma - s)/s} + (CT \|u_{0,2}\|_{H^s})^{(\sigma - s)/s}
$$
to get that
$$
\int_{0}^t \|u_1(\tau)\|_{L^2}d\tau \lesssim T\Big( 1 + (CT \|u_{0,1}\|_{H^s})^{\sigma /s} + (CT \|u_{0,2}\|_{H^s})^{\sigma /s}\Big)\; .
$$
We then use that
$$
\Big( 1 + (CT \|u_{0,1}\|_{H^s})^{(\sigma - s)/s} + (CT \|u_{0,2}\|_{H^s})^{(\sigma - s)/s}\Big)  \leq \Big( 1 + (CT \|u_{0,1}\|_{H^s})^{(\sigma - s)/s} \Big) \Big( 1 + (CT \|u_{0,2}\|_{H^s})^{(\sigma - s)/s}\Big) 
$$
and a second H\"older inequality with $1/p_1 = 1/(2 p_1) +1/(2p_1)$ to get
$$
I \leq I.1 \times I.2 
$$
with
$$
I.1 = \| \Big( 1 + (CT \|u_{0,1}\|_{H^s})^{(\sigma - s)/s} \Big) e^{T\Big( 1 + (CT \|u_{0,1}\|_{H^s})^{\sigma /s}\Big)}\|_{L^{2p_1}_\gamma}
$$
and
$$
I.2 = \| \Big( 1 + (CT \|u_{0,2}\|_{H^s})^{(\sigma - s)/s} \Big) e^{T\Big( 1 + (CT \|u_{0,2}\|_{H^s})^{\sigma /s}\Big)}\|_{L^{2p_1}_\gamma}\; .
$$
As $I.1$ does not depend on $u_{0,2}$ and as $\gamma$ has $\mu$ as a left marginal, we get 
$$
I.1 = \| \Big( 1 + (CT \|u_{0,1}\|_{H^s})^{(\sigma - s)/s} \Big) e^{T\Big( 1 + (CT \|u_{0,1}\|_{H^s})^{\sigma /s}\Big)}\|_{L^{2p_1}_\mu}
$$
which is well defined for all time and all $p_1$ since $\sigma$ has been chosen in $]1/2,2s[$ and thus $\sigma/s < 2$, which is possible since $s>1/4$, and since $\mu$ belongs to $\Sigma$, that is $\mu$ has large Gaussian deviation estimates in $H^s$. Similarly
$$
I.2 = \| \Big( 1 + (CT \|u_{0,2}\|_{H^s})^{(\sigma - s)/s} \Big) e^{T\Big( 1 + (CT \|u_{0,2}\|_{H^s})^{\sigma /s}\Big)}\|_{L^{2p_1}_\nu}
$$
and is well-defined for the same reasons. Note that the bound on $I$ does not depend on $\gamma \in \textrm{Marg} (\mu, \nu)$. We now have the estimate
$$
d_{p, 0} (\mu^t, \nu^t) \leq  C I.1 I.2 \| \|u_{0,1}-u_{0,2}\|_{H^s}\|_{L^{p_2}_\gamma} 
$$
and we conclude by taking the infimum over $\gamma \in \textrm{Marg} (\mu, \nu)$. \end{proof}

\subsection{Stability of invariant measures}

\begin{definition} The measure that is known as $\rho$ in the rest of the paper is an invariant measure on $H^s$ under the flow of BBM, that is, for all measurable set $A$ in $H^s$ and all time $t$, we have 
$$
\rho (\psi(t)^{-1} A) = \rho(A)
$$
or equivalently, for all measurable bounded function $F : H^s \rightarrow \R$, we have 
$$
\E_{\rho} (F) = \E_\rho (F \circ \psi(t) ) \; .
$$
Besides, we assume that $\rho$ admits Gaussian large deviation estimates, that is $\rho \in \Sigma$.
\end{definition}

\begin{remark} At least one measure of this kind exists, as was proved in \cite{dSwav} . For this measure, $s<1/2$.\end{remark}

\begin{proposition}\label{prop-inv} Let $s\in ]1/3,1[$. Let $\mu \in \Sigma_{p_1}$ and $\sigma \in ]\max(1/2,s),\min(3s/2,1)[$. For all $p < p_1$, we have 
$$
d_{p,0} (\mu^t , \rho) \leq C(p_1,t, \mu, \sigma) d_{p_2,s}(\mu,\rho) 
$$
with
\begin{equation}\label{estistab}
C(p_1,t,\mu,\sigma) =  C \Big( 1 + C T^{(\sigma - s)/s} ( \sqrt{\frac{p_1 (\sigma -s )}{s}} + \|u_{0,2}\|_{L^{2p_1(\sigma -s)/s}}^{(\sigma -s)/s})\Big)  e^{c T^2 p_1}
\end{equation}
where $T = 1+|t|$.
\end{proposition}

\begin{proof} The proof begins in the same way as the one of Proposition \ref{prop-cont}. We start from
$$
d_{p, 0} (\rho^t, \mu^t) \leq  C \big\|\Big( 1 + (CT \|u_{0,1}\|_{H^s})^{(\sigma - s)/s} + (CT \|u_{0,2}\|_{H^s})^{(\sigma - s)/s}\Big) e^{c\int_{0}^t \|u_1(\tau)\|_{L^2}d\tau}\big\|_{L^{p_1}_\gamma} \big\|\; \|u_{0,1}-u_{0,2}\|_{H^s}\big\|_{L^{p_2}_\gamma} 
$$
with $\frac1p = \frac1{p_1}+\frac1{p_2}$, and $\sigma > 1/2$ and $\gamma$ has for marginals $\rho$ and $\mu$.

Since $\rho$ is invariant under the flow of BBM, we have $\rho^t = \rho$, hence 
$$
d_{p, 0}(\rho, \mu^t) = d_{p,0} (\rho^t,\mu^t) \; .
$$

We estimate
$$
I = \big\|\Big( 1 + (CT \|u_{0,1}\|_{H^s})^{(\sigma - s)/s} + (CT \|u_{0,2}\|_{H^s})^{(\sigma - s)/s}\Big) e^{c\int_{0}^t \|u_1(\tau)\|_{L^2}d\tau}\big\|_{L^{p_1}_\gamma} \; .
$$
We use a H\"older inequality with $1/p_1 = 1/(2p_1) +1/(2p_1)$:
$$
I \leq I.1 I.2
$$
with 
$$
I.1 = \|\Big( 1 + (CT \|u_{0,1}\|_{H^s})^{(\sigma - s)/s} + (CT \|u_{0,2}\|_{H^s})^{(\sigma - s)/s}\Big)\|_{L^{2p_1}_\gamma}
$$
and 
$$
I.2  = \| e^{c\int_{0}^t \|u_1(\tau)\|_{L^2}d\tau}\|_{L^{2p_1}_\gamma} \; .
$$
We have 
$$
I.1 \leq 1+\|(CT \|u_{0,1}\|_{H^s})^{(\sigma - s)/s}\|_{L^{2p_1}_\gamma} +\| (CT \|u_{0,2}\|_{H^s})^{(\sigma - s)/s}\|_{L^{2p_1}_\gamma}\; .
$$
As $u_{0,1}$ refers to $\rho$ and $u_{0,2}$ to $\mu$, we have 
$$
\|(CT \|u_{0,1}\|_{H^s})^{(\sigma - s)/s}\|_{L^{2p_1}_\gamma} = \|(CT \|u_{0,1}\|_{H^s})^{(\sigma - s)/s}\|_{L^{2p_1}_\rho}
$$
and 
$$
\| (CT \|u_{0,2}\|_{H^s})^{(\sigma - s)/s}\|_{L^{2p_1}_\gamma}= \| (CT \|u_{0,2}\|_{H^s})^{(\sigma - s)/s}\|_{L^{2p_1}_\mu} \; .
$$
As $\rho$ satisfies Gaussian large deviation estimates, we have that
$$
 \|(CT \|u_{0,1}\|_{H^s})^{(\sigma - s)/s}\|_{L^{2p_1}_\rho} \leq C(\rho) T^{(\sigma - s)/s} \sqrt{\frac{p_1 (\sigma -s )}{s}}\; .
$$
For  $\mu$, we have 
$$
\| (CT \|u_{0,2}\|_{H^s})^{(\sigma - s)/s}\|_{L^{2p_1}_\mu} = (C T)^{(\sigma-s)/s} \|u_{0,2}\|_{L^{2p_1(\sigma -s)/s}}^{(\sigma -s)/s} \; .
$$
As $\sigma$ is strictly less than $\frac{3s}2$, we have $2 p_1 (\sigma -s)/s\leq p_1$, hence $\|u_{0,2}\|_{L^{2p_1(\sigma -s)/s}}$ is well defined.

\bigskip

For $I.2$, we use that $x\mapsto e^{Tx}$ is convex, hence from Jensen's inequality,
\begin{eqnarray*}
e^{\int_{0}^t \|u_1(\tau)\|_{L^2}d\tau} &=& e^{T \cdot \frac1T \int_{0}^t \|u_1(\tau)\|_{L^2}d\tau} \\
 & \leq & \frac1T \int_{0}^t e^{T \|u_1(\tau)\|_{L^2}} d\tau \; .
\end{eqnarray*}
We then use Minkowski's inequality
$$
\|\frac1T \int_{0}^t e^{T \|u_1(\tau)\|_{L^2}} d\tau \|_{L^{2p_1}_\gamma} \leq \frac1T \int_{0}^T \|e^{cT \|u_1(\tau)\|_{L^2}}\|_{L^{2p_1}_\gamma} d\tau \; .
$$
As $u_1$ refers to $\rho$ and $\rho$ is invariant under the flow of BBM, we have 
$$
\|e^{cT \|u_1(\tau)\|_{L^2}}\|_{L^{2p_1}_\gamma}= \|e^{cT \|u_1(\tau)\|_{L^2}}\|_{L^{2p_1}_\rho} = \|e^{cT \|u_{0,1}\|_{L^2}}\|_{L^{2p_1}_\rho}\; .
$$
As $\rho$ has Gaussian large deviation estimates, and that for every Gaussian $X$, we have 
$$
\E (e^{r X}) = e^{c(X) r^2}
$$
we get
$$
\|e^{cT \|u_{0,1}\|_{L^2}}\|_{L^{2p_1}_\rho} \leq C(\rho) e^{c(\rho) T^2 p_1}\; .
$$
Finally, considering $\rho$ as a constant of the problem, we get
$$
I \leq \Big( 1 + C T^{(\sigma - s)/s} ( \sqrt{\frac{p_1 (\sigma -s )}{s}} + \|u_{0,2}\|_{L^{2p_1(\sigma -s)/s}}^{(\sigma -s)/s})\Big)  C e^{c T^2 p_1}\; .
$$
As 
$$
d_{0,p} (\rho, \mu^t) \leq C I \|u_{0,1}-u_{0,2}\|_{L^{p_2}_\gamma,H^s}\; ,
$$
we conclude by taking the infimum over $\gamma$. \end{proof}

\begin{remark}As the constant in the estimate is given by \eqref{estistab}:
$$
C(p_1,t,\mu,\sigma) =  C \Big( 1 + C T^{(\sigma - s)/s} ( \sqrt{\frac{p_1 (\sigma -s )}{s}} + \|u_{0,2}\|_{L^{2p_1(\sigma -s)/s}}^{(\sigma -s)/s}\Big)  e^{c T^2 p_1}\; .
$$
If $\mu$ is such that for all $q$
$$
d_{s,q}(\mu, \rho) \leq C^{q} \varepsilon
$$
with $C$ bigger than $1$, and $\varepsilon $ a small parameter,  we get that
$$
d_{0,p}(\mu^t, \rho) \leq (C(t, \sigma))^{p_1} d_{s,p_2} (\mu, \rho) \; .
$$
By taking $p_1 = p_2 = 2p$, we get 
$$
d_{0,p}(\mu^t, \rho) \leq (C'(t, \sigma))^{p} \varepsilon \; .
$$
Hence, thanks to Proposition \ref{prop-alde} with $\alpha =1$, the invariant measure is stable in the set of measure 
$$
\{\mu \; |\; \exists \delta >0 \; ; \; 1_{\|u\|_{H^s}\geq 1}e^{\delta (\ln \|u\|_{H^s})^2} \in L^1_\mu \}\; .
$$
\end{remark}

\begin{remark} We know compare Theorem \ref{theo-result2} with the result of \cite{dSwav}. If $\mu$ and $\nu$ are two Gaussian variables on $H^s$ with respective covariance operators $M_1^*M_1$ and $M^*_2M_2$, then if $X$ has the law $\mu$, $M_2M_1^{-1}X$ has for law $\nu$, or in other words $\nu =( M_2 M_1^{-1})^* \mu$ . Then, if $\gamma$ is defined as 
$$
\gamma(A\times B) = \mu( A \cap M_2^{-1}M_1 B)
$$
we have 
$$
\gamma(A \times H^s) = \mu (A) \mbox{ and } \gamma(H^s \times B) = \mu ( M_2^{-1}M_1 B)= \nu (B)\; ,
$$
in other words $\gamma  \in \textrm{Marg}(\mu, \nu)$. Hence, we have 
$$
d_{s,p} (\mu, \nu) \leq \| \; \|u - v\|_{H^s}\|_{L^p_\gamma} = \|(\mbox{Id}-M_2 M_1^{-1})u\|_{L^p(d\mu(u), H^s)}
$$
and using the operator norm of $\mbox{Id}-M_2 M_1^{-1}$ if it exists, we get
$$
d_{s,p} (\mu, \nu) \leq \|\mbox{Id}-M_2 M_1^{-1}\|_{H^s \rightarrow H^s} \|u\|_{L^p_\mu,H^s}\; .
$$
In the case of \cite{dSwav}, $M_1 = L  = (1- \partial_x^2)^{-1/2}$ and $M_2 = (1+V)^{1/2} L$, thus $M_2 M_1^{-1} = (1+V)^{1/2}$. Hence, the distance between $\mu$ and $\nu$ is controlled by the operator norm of $\mbox{Id} - (1+V)^{1/2}$ and in the case of the operators considered in \cite{dSwav}, it has the size of the perturbation parameter $V$, $|V|$. Hence,
$$
d_{0,p}(\mu^t,\nu^t) \leq C(t,\sigma, p_1,\mu,\nu) |V|
$$
which means that  the result of \cite{dSwav} has been extended.
\end{remark}

\bibliographystyle{amsplain}
\bibliography{bibbbm2} 

\providecommand{\bysame}{\leavevmode\hbox to3em{\hrulefill}\thinspace}
\providecommand{\MR}{\relax\ifhmode\unskip\space\fi MR }
\providecommand{\MRhref}[2]{%
  \href{http://www.ams.org/mathscinet-getitem?mr=#1}{#2}
}
\providecommand{\href}[2]{#2}
\begin{thebibliography}{10}

\bibitem{BCSbouII}
J.~L. Bona, M.~Chen, and J.-C. Saut, \emph{Boussinesq equations and other
  systems for small-amplitude long waves in nonlinear dispersive media. {I}.
  {D}erivation and linear theory}, J. Nonlinear Sci. \textbf{12} (2002), no.~4,
  283--318.

\bibitem{BCSbouI}
\bysame, \emph{Boussinesq equations and other systems for small-amplitude long
  waves in nonlinear dispersive media. {II}. {T}he nonlinear theory},
  Nonlinearity \textbf{17} (2004), no.~3, 925--952.

\bibitem{BTsha}
Jerry~L. Bona and Nikolay Tzvetkov, \emph{Sharp well-posedness results for the
  {BBM} equation}, Discrete Contin. Dyn. Syst. \textbf{23} (2009), no.~4,
  1241--1252.

\bibitem{Bper}
J.~Bourgain, \emph{Periodic nonlinear {S}chr\"odinger equation and invariant
  measures}, Comm. Math. Phys. \textbf{166} (1994), no.~1, 1--26.

\bibitem{BPstat}
R.~Brout and I.~Prigogine, \emph{Statistical mechanics of irreversible
  processes part viii: general theory of weakly coupled systems}, Physica
  \textbf{22} (1956), no.~6–12, 621 -- 636.

\bibitem{BTranI}
Nicolas Burq and Nikolay Tzvetkov, \emph{Random data {C}auchy theory for
  supercritical wave equations. {I}. {L}ocal theory}, Invent. Math.
  \textbf{173} (2008), no.~3, 449--475.

\bibitem{CdScont}
F.~{Cacciafesta} and A.-S. {de Suzzoni}, \emph{{Continuity of the flow of KdV
  with regard to the Wasserstein metrics and application to an invariant
  measure}}, ArXiv e-prints (2013).

\bibitem{dSwav}
Anne-Sophie de~Suzzoni, \emph{Wave {T}urbulence for the {BBM} {E}quation:
  {S}tability of a {G}aussian {S}tatistics {U}nder the {F}low of {BBM}}, Comm.
  Math. Phys. \textbf{326} (2014), no.~3, 773--813.

\bibitem{LRSstat}
Joel~L. Lebowitz, Harvey~A. Rose, and Eugene~R. Speer, \emph{Statistical
  mechanics of the nonlinear {S}chr\"odinger equation}, J. Statist. Phys.
  \textbf{50} (1988), no.~3-4, 657--687.

\bibitem{Pzur}
R.~Peierls, \emph{Zur kinetischen theorie der wärmeleitung in kristallen},
  Annalen der Physik \textbf{395} (1929), no.~8, 1055--1101.

\bibitem{ZFweak}
V.E. Zakharov and N.N. Filonenko, \emph{Weak turbulence of capillary waves},
  Journal of Applied Mechanics and Technical Physics \textbf{8} (1967), no.~5,
  37--40 (English).

\bibitem{Zoninv}
P.~E. Zhidkov, \emph{On invariant measures for some infinite-dimensional
  dynamical systems}, Ann. Inst. H. Poincar\'e Phys. Th\'eor. \textbf{62}
  (1995), no.~3, 267--287. \MR{1335059 (96g:58116)}

\end{thebibliography}
\nocite{*}

\end{document}